\documentclass[11pt]{article}
\usepackage{amsmath, amsfonts, amssymb, mathtools}
\usepackage{graphicx, amsthm, color}
\usepackage{pgfplots, caption, subcaption, tikz, tikz-cd} 
\usepackage{float}
\usepackage{enumitem}
\usepackage{cleveref}
\usepackage{euscript}

\usepackage[scr=boondox]{mathalpha}

\usepackage[T1]{fontenc}
\usepackage{fbb}

\newtheorem{theorem}{Theorem}[section]

\newtheorem{lemma}[theorem]{Lemma}
\newtheorem{proposition}[theorem]{Proposition}
\newtheorem{corollary}[theorem]{Corollary}

\newtheorem{claim}[theorem]{Claim}

\theoremstyle{definition}

\numberwithin{equation}{section}
\newtheorem{remark}[theorem]{Remark}

\newcommand{\R}{\mathbb{R}}
\newcommand{\N}{\mathbb{N}}
\newcommand{\A}{\EuScript A}

\begin{document}

\begin{center}
{\LARGE A convexity criterion via the De Giorgi slope}
\end{center}
\smallskip
\begin{center}
{\Large \textsc{Tahar Z. Boulmezaoud, Aris Daniilidis and Tr\'i Minh L\^e}}
\end{center}

\bigskip

\noindent\textbf{Abstract.} Let $X$ be a Banach space and $f\in\mathcal{C}^1(X)$ be bounded from below. We show that if for some $m\geq 1$, the function $x\mapsto \|\nabla f(x)\|^m$ is convex, then $f$ is convex. We also establish a more general version of this result: if $f$ is continuous and bounded from below, then it is convex, provided $x\mapsto s_f(x)^m$ is convex for some $m\geq 1$, where $s_f$ denotes the (De Giorgi) metric slope of $f$.

\bigskip


\vspace{0.6cm}



\section{Introduction}
In \cite[Corollary~3.17]{BCD_2018} the following convexity criterion was established for $\mathcal{C}^2$-smooth functions in a Hilbert space $\mathcal{H}$:
\begin{itemize}
\item  Let $f\in\mathcal{C}^2(\mathcal{H})$ be bounded below. If $V_2(x)= \|\nabla f(x)\|^2$ is convex, then $f$ is convex.
\end{itemize}
The aim of this work is to extend this criterion to nonsmooth functions defined on a Banach space~$X$. Notice that in principle, this is a challenging task, since the proof of the aforementioned criterion given in \cite{BCD_2018}, depends heavily on the Hilbertian structure as well as on the $\mathcal{C}^2$-smoothness assumption, since it is based on the study of the second order system $\ddot x(t)= \frac{1}{2} \nabla V_2(x)$. \smallskip\newline

As was the case in recent determination results (see \cite{DD_2023, DST_2025, PSV_2021, TZ_2023}) as well as in recent results relating to Monge solutions of the Hamilton-Jacobi equations (\cite{LSZ_2021, LW_2026}), the modulus of the gradient $\|\nabla f\|$ is now replaced by the local (or metric) slope $s_f$ introduced by De Giorgi (see \cite{AGS_2008, GMT_1980} \textit{e.g.}) and defined as follows:
\[
    s_f(x) := \limsup_{y \to x}\, \dfrac{\max \{ f(x) - f(y), 0 \}}{\| x - y \|}\,, \quad\text{for every } x\in X.
\]
Our main result (see forthcoming Theorem~\ref{thm.conti-case}) reads as follows:
\begin{itemize}
    \item Let $f \in \mathcal{C}(X)$ be bounded from below.
If $V_m(x):= s_f(x)^m$ is convex, for some~$m \geq 1$, then $f$ is convex.
\end{itemize}
The proof of the above result will be given in the next section. As expected, the approach differs significantly from the one in \cite{BCD_2018} and borrows from techniques introduced in \cite{GHN_2015, LSZ_2021}. \smallskip\newline
In this work, we denote by $$\mathrm{Crit}(f)=\{s_f=0\}:=\{x\in X:\,s_f(x)=0\}$$ the set of critical points of $f$. Notice that this is the set $\{ \nabla f = 0 \}$ in the smooth case. \smallskip\newline 
A function $f$ is called coercive, if $\overset{\|x\|\to\infty}{\lim}\,f(x)=+\infty$.
 Equivalently, in finite dimensions, the sublevel sets $\{f \leq r\}$, $r\in \R$, are compact.

\medskip

The paper is organized as follows.
Section~2 establishes the main convexity criterion, Section~3 discusses regularity consequences, and Section~4 presents two applications: a Brézis--Ekeland type variational principle and a uniqueness result for eikonal equations in Banach spaces.

\section{Proof of the main result}
For the convenience of the reader, we first present an easy proof in the particular case of a $\mathcal{C}^1$-smooth coercive function in finite dimensions. 

\subsection{A simple proof in the $\mathcal{C}^1$-setting}\label{ss-2.1}

Given a $\mathcal{C}^1$-smooth function $f: \mathbb{R}^d \rightarrow \mathbb{R}$ and $w \in \R^d$, we denote by 
$\gamma_w$ the gradient flow of $f$ starting from $w$, that is,
    \[
        \begin{cases}
            \dot{\gamma}_w(t) = -\nabla f(\gamma_w(t)), & t\geq 0,\\
            \gamma_w(0)=w.
        \end{cases}
    \]
If the curve $\{\gamma_w(t)\}_{t\geq 0}$ lies in a compact set (in particular, if $f$ is coercive), then $\mathrm{dist}(\gamma_w(t), \mathrm{Crit} f) \overset{t \to \infty}{\longrightarrow} \,0$. In what follows, we shall use the following result.
\begin{lemma}\label{lem.ar}
Let $f \in \mathcal{C}^1(\R^d)$ be coercive and $\gamma_w$ the gradient orbit of $f$ starting from $w \in \R^d$.
Then, for every sequence $\{ t_n \}_{n\geq 1}$ with $t_n  \overset{n \to \infty}{\longrightarrow} \infty$, there exist a subsequence $\{t_{n_k}\}_{k\geq 1}$ and $p_w \in \mathrm{Crit}(f)$ such that
$\gamma_w(t_{n_k}) \overset{k \to \infty}{\longrightarrow} \, p_w.$
    
\end{lemma}

\begin{proof}
Along the gradient orbit $\gamma_w$, we have
\[
    \frac{d}{dt} f(\gamma_w(t))
    = \nabla f(\gamma_w(t))\cdot \dot \gamma_w(t)
    = -\|\nabla f(\gamma_w(t))\|^2 \leq 0.
\]
Thus $f(\gamma_w(t))\leq f(w)$ for all $t\geq 0$ and the curve $\gamma_w$ lies in the sublevel set $\{ f \leq f(w)\}$, which is compact, since $f$ is coercive. 
Moreover,
\[
    \int_0^\infty \|\nabla f(\gamma_w(t))\|^2\,dt \,\,\leq \,\, f(w)- \inf\,f \, <\, + \infty\,.
\]
Since $t \mapsto \| \nabla f(\gamma_w(t))\|^2$ is integrable on $[0, +\infty)$, we infer that $\lim_{t \to \infty} \| \nabla f(\gamma_w(t)) \| = 0$. Fix $\{ t_n \}_{n \geq 1}$ with $\lim_{n \to \infty} t_n = +\infty$.
Since the set $\{ f \leq f(w) \}$ is compact, there exist a subsequence $\{ t_{n_k} \}$ and a vector $p_w \in \R^d$ such that $\lim_{k \to \infty} \gamma_w(t_{n_k}) = p_w$. 
Since $\nabla f$ is continuous, we obtain $\nabla f(p_w) = 0$ and so $p_w \in \mathrm{Crit}(f)$, which completes the proof.
\end{proof}

We shall also need the following lemma.

\begin{lemma}\label{lem.crit}
Let $f \in \mathcal{C}^1(\R^d)$ be coercive and assume that $V_2(x) := \| \nabla f(x) \|^2$ is convex. Then $\mathrm{Crit}f$ is nonempty and convex and $f$ is constant there.
\end{lemma}

\begin{proof}
Coercivity of $f$ yields the nonemptiness of $\mathrm{Crit}(f)$.
Notice that the (non-negative) function $V_2$ is convex and $\mathrm{Crit}(f) = \{ V_2 = 0 \}$. Therefore $\mathrm{Crit}(f)$ is convex. Take now any $x,y\in \mathrm{Crit}(f)$. Then for every $s \in (0, 1)$, we have $x+s(y-x) \in \mathrm{Crit}(f)$ and
\[
    \frac{d}{ds} f(x+s(y-x))
    = \underbrace{\nabla f(x+s(y-x))}_{=0}\cdot (y-x)
    =0.
\]
It follows that $f(x)=f(y)$ yielding that $f$ is constant on $\mathrm{Crit}(f)$.
\end{proof}

We are now ready to establish our main result in the particular case where $f$ is $\mathcal{C}^1$-smooth and coercive.

\begin{proposition}\label{thm.C1-coercive}
   Let $f \in \mathcal{C}^1(\R^d)$ be coercive.
    Assume that $V_2(x) := \| \nabla f(x) \|^2$ is convex.
    Then, $f$ is convex.
\end{proposition}

\begin{proof}
Fix $x, y \in \R^d$.
Denote by $\gamma_x, \gamma_y : [0, + \infty) \to \R^d$ the gradient flows starting at $x$ and $y$, respectively, that is, 
\begin{align*}
    \begin{dcases}
        \dot{\gamma_x}(t) = - \nabla f(\gamma_x(t)), \,\, t > 0
        \\
        \gamma_x(0) = x
    \end{dcases}
    \quad \text{ and } \quad
    \begin{dcases}
        \dot{\gamma_y}(t) = - \nabla f(\gamma_y(t)), \,\, t > 0 \\
        \gamma_y(0) = y
    \end{dcases}.
\end{align*}
Set
\begin{align*}
   \mu(t) := & \, \dfrac{\gamma_x(t) + \gamma_y(t)}{2}, \\
   D(t) := & \, \dfrac{1}{2} \Big( f(\gamma_x(t)) + f(\gamma_y(t)) \Big) - f(\mu(t)).
\end{align*}

We first prove that the function $t\mapsto D(t)$ is nonincreasing. 
First observe that 
\begin{equation}\label{Gflow-id}
    \dfrac{d}{dt} f(\gamma_z(t)) = \langle \nabla f(\gamma_z(t)), \dot \gamma_z(t) \rangle = - \| \nabla f(\gamma_z(t)) \|^2 = - V_2(\gamma_z(t)) \quad \text{for $z \in \{ x, y \}$ and $t > 0$.}
\end{equation}
It follows that
\begin{equation}\label{D1}
    D'(t) = - \dfrac{1}{2} \Big( V_2(\gamma_x(t)) + V_2(\gamma_y(t)) \Big) + \left\langle \nabla f(\mu(t)), \dfrac{\nabla f(\gamma_x(t)) + \nabla f(\gamma_y(t))}{2} \right\rangle.
\end{equation}
On the one hand, using the assumption that $V_2$ is convex, we obtain
\begin{equation}\label{D2}
    \| \nabla f(\mu(t)) \| = \sqrt{V_2(\mu(t))} \leq \sqrt{\dfrac{V_2(\gamma_x(t)) + V_2(\gamma_y(t))}{2}}.
\end{equation}
On the other hand, using the convexity of the map $w \mapsto \| w \|^2$, we have
\begin{equation}\label{D3}
    \left\|   \dfrac{\nabla f(\gamma_x(t)) + \nabla f(\gamma_y(t))}{2} \right\| \leq \sqrt{\dfrac{V_2(\gamma_x(t)) + V_2(\gamma_y(t))}{2}}.
\end{equation}
Combining \eqref{D1}, \eqref{D2} and \eqref{D3}, we conclude that $D' \leq 0$ and hence it is nonincreasing.
It follows that
\begin{equation}\label{descent_D}
    D(0) \geq D(t) \quad \text{ for every } t \geq 0.
\end{equation}

\medskip

We next prove that there exists a sequence $\{ t_n \}$ such that $\lim_{n \to \infty} D(t_n) = 0$. 
Applying Lemma~\ref{lem.ar} for $w = x$ and subsequently, for $w = y$, we obtain a sequence $\{ t_n \}_{n\geq 1}$ and $p_x, p_y \in \mathrm{Crit}(f) = \{ \nabla f = 0 \}$ such that
\[
    \gamma_x(t_n) \overset{n \to \infty}{\longrightarrow} \,p_x \quad \text{ and } \quad \gamma_y(t_n) \overset{n \to \infty}{\longrightarrow} \,p_y.
\]
Consequently, thanks to the continuity of $f$ and the fact that $f$ is constant on the convex set $\mathrm{Crit}(f)$ (see Lemma~\ref{lem.crit}), we deduce
\[
    \lim_{n \to \infty} D(t_n) = \dfrac{1}{2} (f(p_x) + f(p_y)) - f \left( \dfrac{p_x + p_y}{2} \right) = 0.
\]

\medskip

To conclude, combining this with~\eqref{descent_D}, we get $D(0) \geq \lim_{n \to \infty} D(t_n) = 0$.
Therefore, 
\begin{align*}
    f \left( \dfrac{x + y}{2} \right) \leq \dfrac{1}{2} (f(x) + f(y)).
\end{align*}
Since $f$ is continuous, we infer that $f$ is convex, which completes the proof.
\end{proof}

\subsection{General case: detecting convexity from the slope in Banach spaces} \label{ss-2.2}

In this subsection we shall prove the general case: $f$ is merely assumed continuous and bounded from below and we work in an arbitrary Banach space. \smallskip\newline
We shall first need the following result.
\begin{lemma}\label{lem.slope_upperbound}
    Let $X$ be a Banach space and let $f \in \mathcal{C}(X)$ be such that $s_f \in \mathcal{C}(X)$.
    Let $\xi: [a, b] \to X$ be an absolutely continuous curve.
    Then, it holds
    \begin{equation}\label{slope_upper}
        f(\xi(s)) - f(\xi(t)) \leq \int_s^t s_f(\xi(\tau)) \| \dot{\xi}(\tau) \|d\tau, \quad \text{ for every } a \leq s \leq t \leq b.
    \end{equation}
\end{lemma}

\begin{proof}
First, consider the case in which $\xi: [a, b] \to X$ a $1$--Lipschitz curve.
In this case, the inequality~\eqref{slope_upper} can be derived by combining \cite[Proposition 4.5]{LSZ_2021} and \cite[Proposition 2.6]{GHN_2015}, via the theory of eikonal equations in metric spaces. 
In the sequel, however, we give a self-contained and direct proof based on the lower  Dini derivative, which avoids the machinery of viscosity solutions in general metric spaces.

\medskip

Consider the function 
\[
    F(t) := f (\xi(t)) + \int_a^t s_f(\xi(\tau)) \, d\tau, \,\, \text{ for } t \in [a, b].
\]
Recall that its \textit{lower right Dini derivative} at $t \in (a, b)$ is defined by:
\[
    D_+ F(t) := \liminf_{h \to 0_+} \dfrac{F(t + h) - F(t)}{h}. 
\]
Fix $t \in (a, b)$ and $\varepsilon > 0$. 
By the definition of slope, there exists $\delta > 0$ such that
\[
    f(\xi(t)) - f(y) \leq (s_f(\xi(t)) + \varepsilon) \| \xi(t) - y \| \quad \text{ for every } y \in B_\delta(\xi(t)).
\]
Using the fact that $\xi$ is 1--Lipschitz, we have, for every  $0 <  h < \min \{\delta, b - t\}$,
\[
    f(\xi(t + h)) - f(\xi(t)) \geq - (s_f(\xi(t)) + \varepsilon) \| \xi(t) - \xi(t + h) \| \geq -  (s_f(\xi(t)) + \varepsilon)h.
\]
Therefore, we obtain
\[
    \dfrac{F(t + h) - F(t)}{h} \geq -  (s_f(\xi(t)) + \varepsilon) + \dfrac{1}{h} \int_t^{t + h} s_f(\xi(\tau)) \,d\tau \quad \text{ for every } 0 < h < \min \{ \delta, b  - t \}.
\]
The continuity of $s_f \circ \xi$ then yields
\[
    D_+ F(t) \geq - (s_f(\xi(t)) + \varepsilon) + s_f(\xi(t)) = - \varepsilon.
\]
Since $\varepsilon > 0$ and $t \in (a, b)$ are chosen arbitrarily, we infer that $D_+ F(t) \geq 0$ for every $t \in (a, b)$.
Therefore $F$ is nondescreasing in $[a, b]$ (see e.g. \cite[Chapter 11, Corollary 4.2]{B_94} or \cite{HT_2006}).
It follows that, for any $a \leq s \leq t \leq b$
\[  
    f(\xi(s)) - f( \xi(t)) \leq \int_a^t s_f (\xi(\tau)) \,d\tau -  \int_a^s s_f(\xi(\tau)) \,d\tau = \int_s^t s_f(\xi(\tau)) \,d\tau.
\]

\medskip

It remains to prove~\eqref{slope_upper} for an arbitrary absolutely continuous
curve $\xi: [a, b] \to X$.
Define the arc-length function
\[
    \ell(r) := \int_a^r \|\dot\xi(\tau)\| \,d\tau, \quad \text{ for } r\in[a,b],
\]
and set $L := \ell(b)$.
The case $L = 0$ is vacuous. 
Assume that $ L> 0 $.
Since $\xi$ is absolutely continuous, $\ell$ is continuous, nondecreasing and satisfies $ \ell'(r)=\|\dot\xi(r)\|$ for a.e. $r \in [a,b]$.
Since $\ell(r_1)=\ell(r_2)$ implies $\xi(r_1)=\xi(r_2)$, the curve
$\eta:[0,L]\to X$ defined by
\[
    \eta(\ell(r)) := \xi(r), \quad \text{  for } r\in[a,b],
\]
is well-defined. 
Furthermore, if $0\le \rho_1\le \rho_2\le L$, we choose $r_1,r_2\in[a,b]$ such that $\ell(r_i) = \rho_i$ for $i \in \{1,2\}$.
Since \(\ell\) is nondecreasing, we may take \(r_1\le r_2\) and so
\[
    \|\eta(\rho_2)-\eta(\rho_1)\|
    =
    \|\xi(r_2)-\xi(r_1)\|
    \le
    \int_{r_1}^{r_2}\|\dot\xi(\tau)\|\,d\tau
    =
    \rho_2-\rho_1.
\]
Thus \(\eta\) is \(1\)-Lipschitz.
Applying the result already proved for $1$-Lipschitz curves, we obtain,
for  any $a\le s\le t\le b$,
\begin{align*}
    f(\xi(s))-f(\xi(t)) = f(\eta(\ell(s))) - f(\eta(\ell(t)) )
    \leq &
    \int_{\ell(s)}^{\ell(t)} s_f(\eta(\rho))\,d\rho \\
    = & \int_s^t s_f(\eta(\ell(\tau)))\ell'(\tau)\,d\tau 
     = 
    \int_s^t s_f(\xi(\tau))\|\dot\xi(\tau)\|\,d\tau,
\end{align*}
where we have used that $\ell'(\tau) = \| \dot{\xi} (\tau) \|$ for a.e $\tau \in [a, b]$.
This completes the proof.
\end{proof}

Before we proceed, let us register the following immediate consequence of Lemma~\ref{lem.slope_upperbound}.
\begin{corollary}\label{cor-ola}
Let $X$ be a Banach space, $f \in \mathcal{C}(X)$ and $s_f \in \mathcal{C}(X)$. Assume that $K$ is a nonempty convex subset of $X$ and $s_f$ is bounded on $K$ by a constant $M> 0$. Then $f$ is $M$-Lipschitz on $K$.
\end{corollary}
\begin{proof}
Given $x,y\in K$, apply~\eqref{slope_upper} for the $1$-Lipschitz curve 
$\xi(t):= y + t \,(\frac{x-y}{\|x-y\|})$, \,\,$t\in \big[\, 0, \|x-y\|\,\big]$. 
\end{proof}

\medskip

We now recall a result from the theory of eikonal equations on complete length spaces. Roughly speaking, the result states that, under the Monge condition $s_u = \ell$, one can construct local curves starting from any point along which $u$ decreases almost optimally. We shall use this curve construction in the sequel.
The following statement is a consequence of Proposition~4.8, Proposition~3.5, and Remark~3.6 of \cite{LSZ_2021}.

\begin{proposition}\label{prop.monge-and-curve}
Let $(X, d)$ be a complete length space and $\Omega \subset X$ an open set.
Assume that $\ell$ is locally uniformly continuous on $X$ and $\inf_\Omega \ell > 0$.
Assume that $u$ is locally Lipschitz on $\Omega$ such that 
\begin{equation}\label{Monge}
	s_u(x) = \ell(x) \quad \text{ for every } x \in \Omega.
\end{equation}
Then, for every $x \in \Omega$, there exists a sufficiently small $\delta > 0$  such that for every $\eta > 0$, there exists a 1--Lipschitz curve $\xi : [0, + \infty) \to  X$ satisfying $\xi(0) = x$ and
\[
	u(x) \geq \int_0^r \ell(\xi (\rho))\, d\rho + u(\xi (r)) - \eta (1 + r), \quad \text{ for every }  0 \leq r \leq R,
\]
where $R > 0$ denotes the exit time of $\xi $ from $B_\delta(x)$, that is, 
\[
	R := \inf \, \big\{ r \geq 0 : \xi (r) \not\in B_\delta(x) \big\}.
\]
\end{proposition}

We obtain easily from the above the following result, which will be used in the sequel.

\begin{corollary}\label{corol.lower-epsi}
Let $X$ be a Banach space and $f \in \mathcal{C}(X)$. Assume that $s_f$ is locally uniformly continuous. 
Then, for every $x \in X$ with $s_f(x) > 0$, there exists $\delta > 0$ such that for every $\varepsilon > 0$ there exists a $1$--Lipschitz curve $\sigma : [0, + \infty ) \to X$ satisfying $\sigma(0) = x$ and 
\[
	f(x) - f(\sigma(r)) \geq \int_0^r s_f(\sigma(\rho)) \,d\rho - \varepsilon \quad \text{ for every } 0 \leq r \leq R,
\]
where $R$ is the exit time of $\sigma$ from $B_\delta(x)$.
\end{corollary}

\begin{proof}
Since the functions $f$ and $s_f$ are continuous and $s_f(x) > 0$, there exists $\delta' > 0$ such that
\[
	m_f(x):=\inf_{B_{\delta'}(x)}\,f > -\infty \qquad \text{and}\qquad    
    0 < \underbrace{\inf_{B_{\delta'}(x)} s_f}_{=: \, m_x} \leq \underbrace{\sup_{B_{\delta'}(x)} s_f}_{=: \, M_x} < + \infty.
\]
Consequently, thanks to Corollary~\ref{cor-ola}, $f$ is Lipschitz on $B_{\delta'}(x)$ with the Lipschitz constant $M_x$.
Applying Proposition~\ref{prop.monge-and-curve} to the case $\Omega = B_{\delta'}(x)$, $u = f$ and $\ell = s_f$, there exists $\delta \in (0, \delta')$ such that, for every $\eta > 0$, there exists a $1$--Lipschitz curve $\xi: [0, + \infty) \to X$ with $\xi(0) = x$ satisfying
\begin{equation}\label{sig_01}
	f(x) \geq \int_0^r s_f(\xi(\rho)) \, d\rho + f(\xi(r)) - \eta(1 + r), \quad \text{ for every } 0 \leq r \leq R_\xi,
\end{equation}
where $R_\xi > 0$ denotes the exit time of $\xi$ from $B_\delta(x)$.
Since $f$ is bounded from below, we get, for every $\eta \leq \min\{1, m_x/2\}$,
\[
	f(x) - m_f(x) + \eta \geq \int_0^{R_\xi} s_f(\xi(\rho)) \, d\rho - \eta R_\xi \geq (m_x - \eta) R_\xi \geq \dfrac{m_x R_\xi}{2}.	
\]
Hence
\[
	R_\xi \leq \dfrac{2(f(x) - m_f(x) + 1)}{m_x}.	
\]
Since $R_\xi$ is bounded by a quantity that depends only on $f$ and $x$, it follows from~\eqref{sig_01} that
\begin{equation}\label{sig_02}
		f(x) \geq \inf_\xi \left\{ \int_0^{R_\xi}  s_f(\xi(\rho)) \, d\rho + f(\xi(R_\xi)) \right \},
\end{equation}
where the infimum is taken over all $1$--Lipschitz curves $\xi: [0, + \infty) \to X$ with $\xi(0) = x$.
Fix $\varepsilon > 0$.
Then, by~\eqref{sig_02}, there exists a $1$--Lipschitz curve $\sigma: [0, + \infty) \to X$ such that $\sigma(0) = x$ and 
\begin{equation}\label{sig_03}
f(x) \geq \int_0^{R_\sigma} s_f(\sigma(\rho))\, d\rho + f(\sigma(R_\sigma))- \varepsilon.
\end{equation}
For any fixed $r \in (0, R_\sigma)$, applying Lemma~\ref{lem.slope_upperbound} to the case $\xi = \sigma$, we have
\begin{equation}\label{sig_04}
	f(\sigma(r)) - f(\sigma(R_\sigma)) \leq \int_r^{R_\sigma} s_f(\sigma(\rho)) \, d\rho.
\end{equation}
Combining~\eqref{sig_03} and \eqref{sig_04}, we obtain
\[
	f(x) \geq \int_0^r s_f(\sigma(\rho)) \, d\rho + f(\sigma(r)) - \varepsilon, \quad \text{ for every } 0 \leq r \leq R_\sigma,
\]
which completes the proof.
\end{proof}

\begin{lemma}\label{lem.m-curve}
    Let $X$ be a Banach space, $m > 1$ and let $f \in \mathcal{C}(X)$ be bounded from below such that $s_f$ is locally uniformly continuous on $X$.
    Then, for any $x \in X$ and $\varepsilon > 0$, there exists $\gamma_x: [0, + \infty) \to X$ such that $\gamma_x(0) = x$,
    \begin{equation}\label{velo_control}
        \| \dot\gamma_x(t) \| \leq s_f^{m - 1}(\gamma_x(t)) \quad \text{ for a.e $t > 0$,}
    \end{equation}
    and
    \begin{equation}\label{descent}
        f(x) - f(\gamma_x(t)) \geq \int_0^t s_f^m(\gamma_x(\tau)) \,d\tau - \varepsilon, \quad \text{ for every } t > 0.
    \end{equation}
\end{lemma}

\begin{proof}
The case $s_f(x) = 0$ is vacuous.
Assume that $s_f(x) > 0$.
We split the proof into two steps.

\medskip

\textbf{Step 1:} \textit{Local existence.}
Fix $x \in X$ and $\varepsilon > 0$.
We prove that there exist $T_{x} > 0$ depending only on $x$ and a curve $\gamma_x: [0, T_x) \to X$ such that
\begin{equation}\label{velo_local}
        \| \dot\gamma_x(t) \| \leq s_f^{m - 1}(\gamma_x(t)) \quad \text{ a.e $t \in (0, T_x)$,}
    \end{equation}
 and
    \begin{equation}\label{descent-local}
        f(x) - f(\gamma_x(t)) \geq \int_0^t s_f^m(\gamma_x(\tau)) \, d\tau - \varepsilon, \quad \text{ for every } t \in (0, T_x).
    \end{equation}

\medskip

First, let $\delta > 0$ be defined as in Corollary~\ref{corol.lower-epsi}.
Notice that
\[
0 < \underbrace{\inf_{y \in B_\delta(x)} s_f(y)}_{:=\,m_x} \leq \underbrace{\sup_{y \in B_\delta(x)} s_f(y)}_{:=\,M_x} < + \infty.
\]
Therefore, thanks to Corollary~\ref{corol.lower-epsi}, there exists a curve $\sigma_x: [0, +\infty) \to X$  with $\sigma_x(0) = x$ such that $\| \dot \sigma_x(r) \| \leq 1$ a.e $r > 0$ and
\begin{equation}\label{eq:2.10}
    f(x) - f(\sigma_x(r)) \geq \int_0^r s_f(\sigma_x(\rho)) \, d\rho - \varepsilon \quad \text{ for every } r \in (0, R_x).
\end{equation}
Here, $R_x$ denotes the exit time of $\sigma_x$ from $B_\delta(x)$:
\[
    R_x := \inf\big\{ r \geq 0 : \sigma_x(r) \not\in B_\delta(x) \big\}.
\]
Since $\sigma_x$ is $1$--Lipschitz, $R_x \geq \delta$.
Define
\[
    \varphi(r)
    := \int_0^r \dfrac{1}{s_f^{m - 1}(\sigma_x(\rho))} \,d\rho,
    \qquad r \in [0,R_x).
\]
Since $\inf_{r \in [0, R_x)} s_f(\sigma_x(r)) \geq m_x> 0$, $\varphi$ is strictly increasing and $\mathcal C^1$ on $(0, R_x)$ with
\[
    \varphi'(r) = \dfrac{1}{s_f^{m - 1}(\sigma_x(r))} \quad \text{ for every } r \in (0, R_x).
\]
Let $\theta := \varphi^{-1} : [0, T_x) \to [0, R_x)$ be the inverse of $\varphi$.
Since $R_x  > \delta/2$, we have
\begin{equation}\label{T_hat}
    T_x = \lim_{r \nearrow R_x} \varphi(r) \geq \varphi(\delta/2) \geq \frac{\delta}{2 M_x^{m-1}} > 0
     \quad \text{ where } M_x := \sup_{y \in B_\delta(x)} s_f(y) < + \infty.
\end{equation}
Set $\gamma_x(t) := \sigma_x(\theta(t))$ for $t \in [0, T_x)$.
Then \(\gamma_x(0)=x\) and for every $t \in (0, T_x)$,  we have
\[
    \theta'(t)
    = \frac{1}{\varphi'(\theta(t))}
    = s_f^{m-1}(\sigma_x(\theta(t)))
    = s_f^{m-1}(\gamma_x(t)).
\]
Therefore
\[
    \|\dot\gamma_x(t)\|
    \leq \|\dot\sigma_x(\theta(t))\|\, \theta'(t)
    \leq s_f^{m-1}(\gamma_x(t))
    \quad \text{for a.e. } t \in (0, T_x).
\]

It remains to verify inequality~\eqref{descent-local} for $t\in (0,T_x)$. To this end, let $t \in (0, T_x)$ and apply the estimate~\eqref{eq:2.10} for $\sigma_x$ at $\theta(t)$, to obtain
\begin{align*}
    f(x)-f(\gamma_x(t))
    = f(x)-f(\sigma_x(\theta(t)))
    \geq & ~ \int_0^{\theta(t)} s_f(\sigma_x(\rho))\,d\rho - \varepsilon
    \\
    = & \,\int_0^t s_f(\sigma_x(\theta(\tau))) \theta'(\tau) \, d\tau - \varepsilon 
    = \, \int_0^t s_f^m(\gamma_x(\tau)) \, d\tau - \varepsilon.
\end{align*}

\medskip

\textbf{Step 2:} \textit{Global existence.}
We shall prove the existence of a curve $\gamma_x: [0, + \infty) \to X$ with $\gamma_x(0) = x$ satisfying~\eqref{velo_local}--\eqref{descent-local} for all $t\in [0,+\infty)$ (that is, we can take $T_x=+\infty$).

\medskip

Fix a continuous strictly increasing function $\kappa : [0,+\infty]\to[0,\varepsilon]$ such that
\[
        \kappa(0)=0, \quad \kappa(T) < \varepsilon\quad\text{for every }T<+\infty
        \quad \text{ and } \quad \kappa(+\infty) = \varepsilon .
\]
For instance, we can take $\kappa(t) = \varepsilon(1 - e^{-t})$ or $\kappa(t)=\frac{2\varepsilon}{\pi}\arctan t$.
Let $\A_x$ be the set of all pairs $(T, \gamma)$, where \(T\in(0,+\infty]\) and $\gamma : [0, T) \to X$ is absolutely continuous such that $\gamma(0) = x$, 
\begin{equation}\label{velocity-T}   \|\dot\gamma(t)\|\leq s_f(\gamma(t))^{m-1}        \quad\text{for a.e. }t\in(0,T),
\end{equation}
and
\begin{equation}\label{descent-T}
    f(x) - f(\gamma(t)) \geq \int_0^t s_f^m(\gamma(s))\,ds - \kappa(T), \quad \text{ for every } t \in (0, T).
\end{equation}
Due to Step 1, we observe that $\A_x$ is nonempty.
Indeed, let $T_x$ be defined as in Step 1.
Applying the local existence result in Step 1 at $x$ with any error $\eta < \kappa(T_x)$ yields the existence of an admissible pair $(T_x, \gamma) \in \A_x$.

\medskip

For any $(T, \gamma), (T', \gamma') \in \A_x$, we write 
\begin{center}
$(T, \gamma) \preceq (T',  \gamma')$\quad if \,\,$T \leq T'$ \,\,and\,\, $\gamma' \big\vert_{[0, T)} = \gamma$.
\end{center}
Observe that $(\A_x, \preceq)$ is partially ordered. 

\begin{claim}\label{claim.chain-upperbound}
    Every chain in $(\A_x, \preceq)$ admits an upper bound.
\end{claim}

\textit{Proof of Claim~\ref{claim.chain-upperbound}.}
Let $\{ (T_i, \gamma_i) \}_{i \in I}$ be a chain in $\A_x$.
Set $T_\ast = \sup_{i \in I} T_i \in (0, + \infty]$.
By total ordering, the curves are compatible: if \(T_i\le T_j\), then $ \gamma_j=\gamma_i$ on $[0, T_i)$.
Thus we may define a curve \(\gamma_*\) on \([0,T_*)\) by setting
\[
        \gamma_\ast(t):=\gamma_i(t) \quad \text{ whenever } t < T_i.
\]
It follows directly that $\gamma_\ast$ is absolutely continuous, $\gamma_\ast(0) = x$ and 
\[
    \| \dot\gamma_\ast(t) \| \leq s_f^{m - 1}(\gamma_\ast(t)) \quad \text{ for a.e } t \in (0, T_\ast).
\]
The curve $\gamma_\ast$ also satisfies the inequality~\eqref{descent-T} for $t \in (0, T_\ast)$.
Indeed, for any fixed $t < T_\ast$, there exists $i \in I$ such that $t < T_i$.
Therefore, we have
\begin{align*}
    f(x) - f(\gamma_\ast(t)) = f(x) - f(\gamma_i(t)) \geq \int_0^t s_f^m(\gamma_\ast(\tau)) \, d\tau - \kappa(T_i) \geq \int_0^t s_f^m(\gamma_\ast(s)) \, ds - \kappa(T_\ast), 
\end{align*}
where we have used that $\kappa(T_i) \leq \kappa(T_\ast)$.
We have proved that $(T_\ast, \gamma_\ast) \in \A_x$.
Therefore, every chain in $(\A_x, \preceq)$ admits an upper bound in $\A_x$. 
\hfill$\Diamond$

\bigskip

Using Claim~\ref{claim.chain-upperbound}, it follows from Zorn's lemma that $(\A_x, \preceq)$ has a maximal element $(T_{\max}, \gamma)$.
We prove that $T_{\max} = + \infty$.
Arguing by contradiction, assume that $T_{\max} < + \infty$.
Notice first that for every $0 \leq s < t < T_{\max}$ we have:
\[
    \| \gamma(s) - \gamma(t) \| \leq \int_s^t \| \dot{\gamma}(\tau) \| \, d\tau \underbrace{\,\leq\,}_{\eqref{velocity-T}} \int_s^t s_f^{m-1}(\gamma(\tau)) \, d\tau\,.
\]
Applying the H\"older inequality for $p=m$ and $q=\frac{m}{m-1}$
(where $\frac{1}{p}+\frac{1}{q}=1$) we obtain:
\[
\int_s^t s_f^{m-1}(\gamma(\tau)) \, d\tau  
\,\leq \,
\left( \int_s^t 1^m d\tau \right)^{\frac{1}{m}} \, \left( \int_s^t s_f^{m-1}(\gamma(\tau))^{\frac{m}{m-1}}  d\tau  \right)^{\frac{m-1}{m}} 
= 
(t-s)^{\frac{1}{m}} \left( \int_s^t s_f^{m}(\gamma(\tau)) d\tau  \right)^{\frac{m-1}{m}}.
\]
Finally, using an estimate similar to~\eqref{descent-T} we obtain:
\[
\int_s^t s_f^m(\gamma(\tau)) d\tau 
\leq 
f(\gamma(s)) - f(\gamma(t)) + \kappa(T_\ast) 
\leq
f(x) - \inf \, f + \varepsilon\,.
\]
Combining the above, we obtain: 
\[
 \| \gamma(s) - \gamma(t) \| \, \leq \, C\,(t-s)^{\frac{1}{m}}\quad\text{where } \, C:=\left( f(x) - \inf \,f + \varepsilon \right)^{\frac{m-1}{m}}\,.
\]
Since $X$ is complete, the limit $\bar x  := \gamma(T_{\max}) = \lim_{t \nearrow T_{\max} } \gamma(t) $ exists.

\medskip

If $s_f(\bar x) = 0$, take any $\alpha > 0$ and define $       \widetilde\gamma:[0,T_{\max}+\alpha)\to X
$
by
\[
\widetilde\gamma(t) :=
        \begin{cases}
            \gamma(t), & 0\le t \leq T_{\max},\\
            \phantom{t}\bar x\,, & T_{\max}\le t < T_{\max}+\alpha .
        \end{cases}
\]
One can directly check that $(T_{\max} + \alpha, \widetilde\gamma) \in \A_x$, which contradicts the maximality of $(T_{\max}, \gamma)$.

\medskip

Consider the case $s_f(\bar x) > 0$.
Let $T_{\bar x} > 0$ be defined as in Step 1 (at $\bar x$).
Fix $0 < \alpha < T_{\bar x}$.
Since $\kappa$ is strictly increasing, we may choose
\[
    0 < \eta < \kappa(T_{\max}  + \alpha) - \kappa(T_{\max}).
\]
Thanks to Step 1, there exists a curve $\gamma_{\bar x} : [0, T_{\bar x}) \to X$ such that $\gamma_{\bar x}(0) = \bar x$ and it satisfies~\eqref{velo_local}--\eqref{descent-local} with the error $\eta$.
Define $\widetilde\gamma:[0,T_{\max} + \alpha)\to X$ by
\[      \widetilde\gamma(t):=
        \begin{dcases}
            \phantom{tri}\gamma(t), & 0\le t\le T_{\max},\\
            \gamma_{\bar x}(t-T_{\max}), & T_{\max}\le t < T_{\max}+\alpha .
        \end{dcases}
\]

We shall prove that $(T_{\max} + \alpha, \widetilde\gamma) \in \A_x$.
Observe first that $\widetilde\gamma$ satisfies the estimate~\eqref{velocity-T} for a.e $t \in (0, T_{\max} + \alpha)$.
It remains to check that it satisfies~\eqref{descent-T} for every $t \in (0, T_{\max} + \alpha)$.
It suffices to consider the case $t = T_{\max} + t^\prime$ for any $t^\prime \in [0, \alpha)$.
Using the monotone convergence theorem and the continuity of $f$, we obtain
\begin{align*}
    f(x) - f(\widetilde\gamma(T_{\max})) = & \, f(x) - \lim_{t \nearrow T_{\max}} f(\gamma(t)) \\
    \geq & \, \lim_{t \nearrow  T_{\max}}  \int_0^{t} s_f^m(\widetilde\gamma(\tau)) \, d\tau - \kappa(T_{\max}) \geq \int_0^{T_{\max}} s_f^m(\widetilde\gamma(\tau)) \, d\tau - \kappa(T_{\max} + \alpha).
\end{align*}
Then, we have, for any $t = T_{\max} + t^\prime \in (T_{\max}, T_{\max} + \alpha)$
\begin{align*}
    f(x) - f(\widetilde\gamma(t)) = &\, f(x) - f(\gamma_{\bar x}(t^\prime)) \\
    = & \, f(x) - f(\gamma(T_{\max})) + f(\bar x) - f(\gamma_{\bar x}(t^\prime)) \\
    \geq & \, \int_0^{T_{\max}} s_f^{m}(\gamma(\tau)) \, d\tau - \kappa(T_{\max}) +\int_0^{t^\prime} s_f^{m}(\gamma_{\bar x}(\tau)) \, d\tau - \eta \\
    = & \, \int_0^t s_f^{m}(\widetilde \gamma(\tau)) \, d\tau - (\eta + \kappa(T_{\max}))
    \\
    \geq & \, \int_0^t s_f^{m}(\widetilde \gamma(\tau)) \, d\tau - \kappa(T_{\max} + \alpha), 
\end{align*}
where the last inequality follows from the choice of $\eta$.
Thus, we have $(T_{\max} + \alpha, \widetilde\gamma) \in \A_x$, again contradicting the maximality of $(T_{\max}, \gamma)$.

\medskip

In conclusion, we have proved that the maximal element $(T_{\max}, \gamma)$ of $(\A_x, \preceq)$ satisfies $T_{\max} = + \infty$. 
Using the fact that $\kappa(+ \infty) = \varepsilon$, we conclude that the curve $\gamma$ satisfies the desired estimates.
This proves Lemma~\ref{lem.m-curve}.
\end{proof}

Before proceeding to our main result, we shall also need the following proposition.

\begin{proposition} \label{prop-god}
Let $X$ be a Banach space. Assume $f \in \mathcal{C}(X)$ and
$s_f: X \to \R$ is convex. Then, $s_f$ is continuous (and consequently, locally Lipschitz). 
\end{proposition}

\begin{proof} Since $f$ is continuous, writing $s_f(x)$ as 
\[
s_f(x) := \lim_{n\to \infty} \,\underbrace{\sup_{0 < \|h\|<\frac{1}{n}}\,\left(\dfrac{\max \{ f(x) - f(x+h), 0 \}}{\| h \|}\right)}_{:=\phi_n(x) \,\,\text{ (lsc function of $x$)}}\,, \quad\text{for every } x\in X,
\]
we deduce that $s_f$ is a Baire-$2$ function and consequently Baire-measurable
\footnote{Here, a Baire-$1$ function is a pointwise limit of continuous functions and a Baire-$2$ function is a pointwise limit of Baire-$1$ functions. In metric spaces, lower semicontinuous functions are Baire-$1$. A real-valued function $g:X\to\mathbb R$ is Baire-measurable if, for every nontrivial open interval $I\subset\mathbb R$, there is an open set $O\subset X$ such that the symmetric difference $g^{-1}(I)\,\triangle\, O$ is a countable union of nowhere dense subsets of $X$.}.
The result follows directly from \cite[Theorem~6]{M_1964}. \end{proof}

\begin{remark}[continuity of convex functions]\label{rem-god} It is well-known that in a finite dimensional space, every convex function $f$ with values in $\R$ is continuous. In every infinite dimensional Banach space, there exist convex (even linear) real-valued functions, which are discontinuous. In all these examples, the functions fail to be Baire-measurable.  \smallskip\newline
One can provide a direct proof (communicated to us by G. Godefroy) of the fact that every Baire-measurable convex function $f$ from a Banach space $X$ to $\R$ is continuous. Indeed, let $x\in X$ be an arbitrary point. By a standard argument (change of coordinates), we can assume that $x=0$ and $f(0)=0$. Setting ${C_n:=\{f\leq n\}}$, for $n\ge 1$, we have that $C_n$ is convex, absorbing and Baire (by assumption). Since $X=\bigcup_{n\geq 1}\,C_n$, there exists some $n_0\in\N$ such that $C_{n_0}$ is of second category. By \cite[Lemme~VI.4.2]{God} $C_{n_0}$ is a neighborhood of $0$ and continuity of $f$ follows. \smallskip\newline
Let us finally recall that completeness of the space $X$ is essential: indeed, taking $X=(c_{00}(\N),\|\cdot\|_2)$ (the space of eventually null sequences equipped with the $2$-norm) we see that the lower semicontinuous convex function $f(x)=\|x\|_1:=\sum_{n\in\N}|x_n|$,
for $x=(x_n)_n \in c_{00}(\N)$, is nowhere continuous.
\end{remark}
\medskip

We are now ready to establish the main result of this work.

\begin{theorem}\label{thm.conti-case} {\normalfont (convexity criterion via slope)}
    Let $X$ be a Banach space and let $f \in \mathcal{C}(X)$ be bounded from below.
    Assume that for some~${m \geq 1}$, the function $V_m := s_f^m: X \to \R$ is convex. Then, $f$ is convex.
\end{theorem}

\begin{remark}
The assumptions in Theorem~\ref{thm.conti-case} are essential and the
criterion is not reversible.\smallskip\newline
\phantom{t}\textbf{(i).} The continuity assumption on $f$ cannot, in general, be
    weakened to lower semicontinuity. Indeed, let $f:\mathbb R\to\mathbb R$
    be given by
    \[
        f(x)=
        \begin{dcases}
            \phantom{-}x^2, & x\neq 0,\\
            -1\,, & x=0.
        \end{dcases}
    \]
    Then $f$ is lower semicontinuous, bounded from below and nonconvex.
    However $s_f^2(x)=4x^2$ is smooth and convex.
\medskip\newline
\phantom{t}\textbf{(ii).} The assumption that $f$ is bounded from below is essential and cannot be omitted. For instance, the function ${f(x,y)=x^2-y^2}$ is continuous and nonconvex but is not bounded from below. Moreover, ${\| \nabla f(x, y)\|^2 =4(x^2+y^2)}$  is convex. 
    See also~\cite[Remark~3.18]{BCD_2018} for other examples.
\medskip\newline
\phantom{t}\textbf{(iii).} The converse implication in Theorem~\ref{thm.conti-case} is false.
    Namely, convexity of $f$ does not imply convexity of $s_f^m$. 
    Consider
    \[
        f(x)=\sqrt{1 + x^2}, \quad  x\in\mathbb R.
    \]
    Then $f$ is convex, $C^1$, and bounded from below. However,
    \[
        s_f^m(x)=|f'(x)|^m
        =
        \frac{|x|^m}{(1 + x^2)^{m/2}},
        \quad x\in\mathbb R,
    \]
    is not convex for any $m\geq 1$.
\end{remark}

\begin{proof}[Proof of Theorem~\ref{thm.conti-case}]
Notice that if the function $s_f$ is convex, then so is the function $s_f^m$, for every $m\geq 1$. Therefore, it suffices to establish the result for $m > 1$. \smallskip\newline
Notice that according to Proposition~\ref{prop-god}, the convex function $s_f^m$ is continuous and consequently, locally Lipschitz. 
Fix $x, y \in X$ and $\varepsilon > 0$.
Thanks to Lemma~\ref{lem.m-curve}, for each $z \in \{x, y\}$, there exists a curve $\gamma_z: [0, + \infty) \to X$ such that $\gamma_z(0) = z$, 
\begin{equation}\label{velo_control}
    \| \dot{\gamma_z}(t) \| \leq s_f^{m - 1}(\gamma_z(t)) \quad \text{ for  a.e } t > 0,
\end{equation}
and
\begin{equation}\label{eq.gamma_epsi}
    f(z) - f(\gamma_z(t)) \geq \int_0^t V_m(\gamma_z(\tau)) \,d\tau - \varepsilon \quad \text{ for every } t > 0.
\end{equation}
Set 
\begin{align*} 
    \mu(t) := & \,\dfrac{\gamma_x(t) + \gamma_y(t)}{2} \qquad \text{and} \qquad D(t) := \, 
    \dfrac{1}{2} \Big( f(\gamma_x(t)) + f(\gamma_y(t)) \Big) - f(\mu(t)).
\end{align*}

We first prove that for every $t > 0$, one has $D(t) - 2 \varepsilon \leq D(0)$.
Indeed, it follows from \eqref{eq.gamma_epsi} that
\[
    D(t) - D(0) \leq - \dfrac{1}{2}\int_0^t \Big( V_m(\gamma_x(\tau)) + V_m(\gamma_y(\tau)) \Big) \, d\tau + 2 \varepsilon + f(\mu(0)) - f(\mu(t)).
\]
Applying Lemma~\ref{lem.slope_upperbound} to the curve $\mu: [0, t] \to X$, we obtain 
\[
    f(\mu(0)) - f(\mu(t)) \leq \int_0^t s_f(\mu(\tau)) \| \dot{\mu}(\tau) \| \, d\tau.
\]
H\"older's inequality then yields 
\[
\int_0^t s_f(\mu(\tau)) \| \dot{\mu}(\tau) \| \,d\tau  \leq \left( \int_0^t V_m(\mu(\tau)) \,d\tau \right)^\frac{1}{m} \left( \int_0^t \| \dot\mu(\tau) \|^\frac{m}{m - 1} \,d\tau \right)^{\frac{m - 1}{m}}.
\]
On the one hand, it follows from the convexity of $V_m$ that
\[
    \int_0^t V_m(\mu(\tau)) \,d\tau \, \leq \, \dfrac{1}{2} \int_0^t V_m(\gamma_x(\tau)) + V_m(\gamma_y(\tau)) \, d\tau.
\]
On the other hand, recall that for each $z \in \{ x, y \}$, one has $\| \dot\gamma_z \| \leq s_f^{m - 1}(\gamma_z(t))$ for a.e $t > 0$.
Hence the convexity of the map $w \mapsto \| w \|^\frac{m}{m - 1}$ yields
\[
    \int_0^t \| \dot\mu (\tau) \|^\frac{m}{m - 1} \,d\tau \leq \dfrac{1}{2} \int_0^t \Big( \| \dot\gamma_x(\tau) \|^\frac{m}{m - 1} + \| \dot\gamma_x(\tau) \|^\frac{m}{m - 1} \Big) \, d\tau \leq \dfrac{1}{2} \int_0^t V_m(\gamma_x(\tau)) + V_m(\gamma_y(\tau)) \, d\tau.
\]
Therefore,
\[
    f(\mu(0)) - f(\mu(t)) \leq \dfrac{1}{2} \int_0^t V_m(\gamma_x(\tau)) + V_m(\gamma_y(\tau)) \, d\tau
\]
which yields
\begin{equation}\label{D_differ_epsi}
    D(t) - D(0) \leq 2\varepsilon.
\end{equation}

\medskip

We now quantify the difference between $f(\mu(t))$ and $f(\gamma_z(t))$ for each $z \in \{x, y\}$.
Fix $z \in \{ x, y \}$ and $t > 0$.
Applying Lemma~\ref{lem.slope_upperbound} to the straight-line segment connecting $\gamma_z(t)$ and $\mu(t)$, we get
\begin{equation}\label{f-differ}
\begin{split}
    f(\gamma_z(t)) - f(\mu(t)) 
    \leq & \,
    \int_0^1 s_f(\underbrace{\mu(t) + \tau(\gamma_z(t) - \mu(t))}_{\xi(\tau)}) \, \| \underbrace{\gamma_z(t) - \mu(t)}_{\dot{\xi}(\tau)} \| \, d\tau 
    \\
    = & \, \int_0^1V_m(\mu(t) + \tau(\gamma_z(t) - \mu(t))^{1/m }  \| \gamma_z(t) - \mu(t) \| \, d\tau
    \\
    \leq & \, \left(  V_m(\gamma_x(t)) + V_m(\gamma_y(t)) \right)^{1/m}\| \gamma_z(t) - \mu(t) \|,
\end{split}
\end{equation}
where convexity of $V_m$ is used to obtain the last inequality. 
Furthermore, it follows from~\eqref{velo_control} and H\"older's inequality that
\[
    \| \gamma_z(t) - z \| \leq \int_0^t \| \dot{\gamma_z}(\tau) \| \, d\tau \leq t^{1/m} \left( \int_0^t s_f^m(\gamma_z(\tau)) \, d\tau \right)^{\frac{m - 1}{m}} \leq C_z t^{1/m},
\]
where 
$C_z := \big( f(z) - \inf \, f + \varepsilon \big)^{(m - 1)/m}$.
Therefore, there exists a constant $C > 0$, depending on $f, x, y$ and $\varepsilon > 0$, such that
\begin{equation}\label{gama_mu}
    \| \gamma_z(t) - \mu(t) \| \leq C t^{1/m} \quad \text{ for every } t > 0.
\end{equation}
Combining~\eqref{f-differ} and \eqref{gama_mu}, we arrive at
\[
    f(\gamma_z(t)) - f(\mu(t)) \leq C t^{1/m}\left(  V_m(\gamma_x(t)) + V_m(\gamma_y(t)) \right)^{1/m}.
\]
By a similar argument, we obtain
\begin{equation}\label{f_differ_abs}
    \big|  f(\gamma_z(t)) - f(\mu(t)) \big| \leq C t^{1/m}\left(  V_m(\gamma_x(t)) + V_m(\gamma_y(t)) \right)^{1/m} \quad \text{ for every } t > 0 \text{ and } z \in \{x, y\}.
\end{equation}

To continue, we need the following fundamental fact.

\begin{claim}\label{claim.tn}
    Let $h\in L^1(\R_+,\R_+)$. Then there exists a sequence 
    $\{ t_n \}_n$ such that
    \[
        t_n \overset{n\to\infty}{\longrightarrow} \infty \quad \text{  and } \quad t_n\,h(t_n) \overset{n\to\infty}{\longrightarrow}  0.
    \]
\end{claim}

\textit{Proof of Claim~\ref{claim.tn}.}
Arguing by contradiction, assume that there exists $\varepsilon > 0$ and $T > 0$ such that
\begin{align*}
    t h(t) > \varepsilon \,\, \text{ for every } t > T.
\end{align*}
It follows that 
\[
    \int_T^\infty h(t) \, dt \geq \varepsilon \int_T^\infty \dfrac{dt}{t} = + \infty,
\]
which contradicts the fact that $h$ is integrable. 
\hfill $\Diamond$

\medskip

Thanks to~\eqref{eq.gamma_epsi}, we observe that
\[
    \int_0^\infty V_m(\gamma_x(t)) + V_m(\gamma_y(t)) \, dt< + \infty.
\]
Therefore, applying Claim~\ref{claim.tn} to the case $h(t) = V_m(\gamma_x(t)) + V_m(\gamma_y(t))$, there exists $\{ t_n \}$ such that $t_n \overset{n \to \infty}{\longrightarrow} \infty$ and 
\[
    t_n \Big( V_m(\gamma_x(t_n)) + V_m(\gamma_y(t_n)) \Big) \to 0 \text{ as } n \to \infty.
\]
Substituting $t = t_n$ into \eqref{f_differ_abs}, we obtain
\[
    |D(t_n)| \leq \dfrac{1}{2} \Big( |f( \gamma_x(t_n) ) - f( \mu(t_n) ) | + |f( \gamma_y(t_n) ) - f( \mu(t_n) ) |  \Big) \to 0 \quad \text{ as } n  \to \infty.
\]
We have proved that there exists a sequence $\{t_n\}$ such that
\begin{equation}\label{limD_zero}
    \lim_{n \to \infty} D(t_n) = 0.
\end{equation}

\medskip

To conclude, combining~\eqref{D_differ_epsi} and \eqref{limD_zero}, we get $  D(0) \geq  - 2\varepsilon.$
Since $\varepsilon > 0$ is chosen arbitrarily, we deduce $D(0) \geq 0$, equivalently
\[
    f \left( \dfrac{x + y}{2} \right) \leq \dfrac{1}{2} (f(x) + f(y)).
\]
Therefore, $f$ is convex, which completes the proof.
\end{proof}

\section{Inducing regularity from the slope}
In this section, we complete the previous result by observing that in the Hilbert case, continuity of the slope mapping $s_f$ induces some extra regularity on $f$. This is the aim of the following proposition.

\begin{proposition}\label{prop.conti_to_C1}
    Let $\mathcal{H}$ be a Hilbert space and $f : \mathcal{H} \to \R$ be a lower semicontinuous convex function with continuous slope $s_f$.
    Then, $f \in \mathcal{C}^1(\mathcal{H})$.
\end{proposition}

\begin{proof}
    We first prove that the convex subdifferential $\partial f$ is a singleton at every point in $\mathcal{H}$.
    Indeed, arguing by contradiction, we assume that there exists $\bar x \in \mathcal{H}$ such that $\partial f(\bar x)$ is not a singleton. 
    Then, there exists $\xi_0 \in \partial f(\bar x)$ with $\xi_0 \neq \xi_{\bar x}$, where $\{ \xi_{\bar x} \} = \mathrm{argmin} \big\{ \| \xi \| : \xi \in \partial f(\bar x) \big\}$.
    Consequently,
    \begin{equation}\label{eq.contra}
        \| \xi_0 \| > \|\xi_{\bar x}\| =  \min_{\xi \in \partial f(\bar x)} \| \xi \| = s_f(\bar x).
    \end{equation}
    Denote $e := \xi_0 / \| \xi_0 \|$.
    For any fixed $t > 0$, it follows from the monotonicity of $\partial f$ that
    \[
        \langle \eta - \xi_0, te \rangle \geq 0 \,\, \text{ for every } \eta \in \partial f(\bar x + te).
    \]
    Hence, using the Cauchy--Schwarz inequality and the fact that $\| e \| = 1$, we get
    \[
        \| \eta \| \geq \langle \eta, e \rangle \geq \langle \xi_0, e \rangle = \| \xi_0 \| \,\, \text{ for every } \eta \in \partial f(\bar x + te). 
    \]
    Taking the infimum with respect to $\eta$, we obtain
    \[
        s_f(\bar x + te) \geq \| \xi_0 \| \,\, \text{ for every } t > 0, 
    \]
    which, thanks to the continuity of slope, leads to
    \[
        s_f(\bar x) = \lim_{t \to 0^+} s_f(\bar x + te) \geq \| \xi_0 \|,
    \]
    which contradicts \eqref{eq.contra}. 
    Therefore, $\partial f$ is singleton at every point in $\mathcal{H}$. 

    \medskip

    To conclude, we will prove that $\nabla f(\cdot)$ is continuous.
    Let $\{ x_n \} \subset \mathcal{H}$ and $x \in \mathcal{H}$ be such that $x_n \to x$ as $n \to \infty$.
    It follows that $\| \nabla f (x_n) \| \to \| \nabla f (x)\|$ as $n \to \infty$.
    Hence $\{\nabla f (x_n) \}_n$ is bounded and so there exist $\xi \in \mathcal{H}$ such that, up to a subsequence, $\nabla f(x_n) \rightharpoonup  \xi$ as $n \to \infty$.
    One can show that $\xi \in \partial f(x)$ and hence $\xi = \nabla f(x)$.
    Since $\{ \nabla f(x_n) \}$ converges weakly and in norm to $\nabla f(x)$, we infer that it strongly converge to $\nabla f(x)$.
    Proposition~\ref{prop.conti_to_C1} is proven.
\end{proof}

As a direct consequence of Theorem~\ref{thm.conti-case} and Proposition~\ref{prop.conti_to_C1}, we obtain the following result.
\begin{corollary}\label{corol.C1+conv}
    Let $\mathcal{H}$ be a Hilbert space. 
    Let $f \in \mathcal{C}(\mathcal{H})$ be a bounded from below such that $s_f^m$ is convex, for some $m\geq 1$.
    Then, $f$ is a $\mathcal{C}^1$ convex function on $\mathcal{H}$.
\end{corollary}

\bigskip

We shall now show that the above result is essentially optimal, in the sense that we cannot obtain better regularity on $f$, even if we assume more regularity on $s_f$.

\begin{proposition}\label{prop.higher-regu}
Let $\beta \in (0, 1]$.
    Then, the following assertions hold true: 
    \begin{itemize}
        \item[$(i)$] If $f \in \mathcal{C}^1(\R)$ and its slope $|f'|$ is locally $\beta$--H\"older, then $f$ is of class $\mathcal{C}^{1, \beta}_\mathrm{loc}(\mathbb R)$.
        \item[$(ii)$] There exists a convex $\mathcal{C}^1$ function $f: \R^2 \to [0, + \infty)$ such that $\| \nabla f \|$ is locally $\beta$--H\"older but the gradient $\nabla f$ fails to be $\beta$--H\"older around the origin. 
    \end{itemize}
\end{proposition}
\begin{proof}
$(i)$
Fix $\bar x \in \mathbb R$ and choose $\delta > 0$ sufficiently small.
Set $I := [\bar x-\delta,\bar x+\delta]$.
Since $|f'|$ is locally $\beta$--H\"older, there exists $C_I> 0$ such that
\[
    \big||f'|(x)-|f'|(y)\big|
    \leq C_I |x-y|^\beta
    \qquad \text{for every } x,y\in I.
\]
We claim that $f'$ is $\beta$--H\"older on $I$.
Fix $x,y\in I$.
The case $f'(x)\cdot f'(y) = 0$ is immediate.
If $f'(x)$ and $f'(y)$ have the same sign, then
\[
    |f'(x)-f'(y)|
    =
    \big||f'|(x)-|f'|(y)\big|
    \leq C_I |x-y|^\beta .
\]

It remains to consider the case in which $f'(x)$ and $f'(y)$ have opposite signs. 
By the continuity of $f'$, there exists $z$ between $x$ and $y$ such that
$f'(z)=0$. Hence
\[
    |f'(x)|
     \leq C_I |x-z|^\beta \text{ and }
    |f'(y)|
    \leq C_I |y-z|^\beta
\]
The concavity of the map $r \mapsto r^\beta$ implies that
\[
    |x-z|^\beta+|y-z|^\beta
    \leq 2^{1 - \beta} (|x - z| + |y - z|)^\beta = 2^{1-\beta}|x-y|^\beta.
\]
Therefore,
\[
    |f'(x)-f'(y)|
    \leq |f'(x)|+|f'(y)|
    \leq2^{1 -\beta} C_I |x - y|^\beta.
\]
Thus $f'\in \mathcal{C}^{0,\beta}(I)$. Since $\bar x$ was arbitrary, we conclude that
$f\in \mathcal{C}^{1,\beta}_{\mathrm{loc}}(\mathbb R)$.

\medskip
$(ii)$ 
Fix $0<\beta\leq 1$ and define
\[
    f(x,y)=e^x+\frac{1}{1+\beta/2}|y|^{1+\beta/2},
    \qquad \text{ for } (x,y)\in \mathbb R^2 .
\]
Observe first that $f\geq 0$, $f\in \mathcal{C}^1(\mathbb R^2)$ and is convex.
Also a direct computation yields
\[ 
    \nabla f(x,y) = \big(e^x,\operatorname{sgn}(y)|y|^{\beta/2}\big) \quad 
    \text{ and } \quad V(x, y) := \| \nabla f(x, y) \| = \big(e^{2x}+|y|^\beta\big)^{1/2}.
\]
It is straightforward to see that $\nabla f$ is locally $(\beta/2)$--H\"older and fails to be $\beta$--H\"older around the origin.

\medskip

We claim that $V$ is locally $\beta$--H\"older.
Fix $(\bar x,\bar y)\in\mathbb R^2$ and let $K\subset\mathbb R^2$ be a compact neighborhood of $(\bar x,\bar y)$.
For $z=(x,y)$ and $z'=(x',y')\in K$, we have
\[
\begin{aligned}
    |V(z)-V(z')|
    &=
    \Big|
        \big(e^{2x}+|y|^\beta\big)^{1/2}
        -
        \big(e^{2x'}+|y'|^\beta\big)^{1/2}
    \Big|
    = 
    \frac{\big|e^{2x}-e^{2x'}+|y|^\beta-|y'|^\beta\big|}
        {
        \big(e^{2x}+|y|^\beta\big)^{1/2} + \big(e^{2x'}+|y'|^\beta\big)^{1/2}
        }.
\end{aligned}
\]
Since $e^{2x}\geq 1$ for all $x\in\mathbb R$, we get
\[
    |V(z)-V(z')|
    \leq
    \frac12 |e^{2x}-e^{2x'}|
    +
    \frac12 \bigl||y|^\beta-|y'|^\beta\bigr|.
\]
Note that $x \mapsto e^{2x}$ is smooth in $\R$ and for $0<\beta\leq 1$, the map $t\mapsto |t|^\beta$ is globally $\beta$--H\"older on $\R$.
Therefore, there exists $C_K > 0$ such that 
\[
    |V(z)-V(z')|
    \leq
    C_K\big(|x-x'|^\beta+|y-y'|^\beta\big)
    \leq
    C_K |z-z'|^\beta
    \qquad\text{for all } z,z'\in K.
\]
Hence $V=\|\nabla f\|$ is locally $\beta$--H\"older.
This completes the proof.
\end{proof}

\begin{remark} $\, $
\begin{itemize}
\item[(i)]
For $\beta \in (1/2,1]$, the function in Proposition~\ref{prop.higher-regu}--(ii) can
even be chosen to be bounded from below and to have convex squared slope. 
Indeed, fix $p\in[1,2\beta)$ and set
\[
q(y) \,:= \, \mathrm{sgn}(y)|y|^{p/2}\sqrt{1 + |y|^p} \quad \text{ and } \quad f(x, y) \,:=\,  e^x  + \int_0^y q(t) \, dt.
\]
Then $f\geq 0$, $f\in \mathcal{C}^1(\mathbb R^2)$ and $f$ is convex, since $e^x$ is
convex and $q$ is nondecreasing.
A direct computation yields
\[
\nabla f(x,y)=(e^x,q(y)) \quad \text{ and } \quad \| \nabla f(x, y) \|^2 = e^{2x} +|y|^p + |y|^{2p}.
\]
Hence $\| \nabla f \|^2$ is convex and the slope $\| \nabla f \|$ is locally Lipschitz, whose proof is similar to that of  Proposition~\ref{prop.higher-regu}-(ii).

\medskip

However, since $p/2<\beta$, we have
\[
\frac{|q(y)-q(0)|}{|y|^\beta}
=
|y|^{p/2-\beta}\sqrt{1 + |y|^p}
\longrightarrow +\infty
\qquad\text{as } y\to0.
\]
Therefore $\nabla f$ is not $\beta$--H\"older at the origin.

\item[(ii)] 
Proposition~\ref{prop.conti_to_C1} shows that H\"older regularity of $\|\nabla f\|$ cannot, in general, be transferred to H\"older regularity of $\nabla f$.  
This is in sharp contrast to the rigidity phenomena for the classical eikonal equation
\[
        \|\nabla f\| \equiv 1,
\]
as studied by Caffarelli--Crandall \cite{CC_2010} and Ignat \cite{I_2025}.  
The point is that these results use the constant slope to establish the $\mathcal{C}^{1, 1}_\mathrm{loc}$ estimates for solutions, not merely regularity of the scalar field $\|\nabla f\|$.  
Assertion $(ii)$ shows that once the identity $\|\nabla f\|\equiv1$ is replaced by the weaker assumption ${\|\nabla f\|\in \mathcal{C}^{0,\beta}_{\mathrm{loc}}}$, the regularity of the slope no longer controls the regularity of the gradient.
\end{itemize}
\end{remark}

\section{Consequences}

We conclude with two applications of the preceding results.
The first concerns a Br\'ezis--Ekeland type variational principle while the second yields a uniqueness result for a class of eikonal equations.

\subsection{A variational principle for power--law gradient flows}

The variational characterization of gradient flows goes back to 
Br\'ezis and Ekeland, see~\cite{BE_76,BE_76_02}.
Further variational characterizations of gradient flows, as well as extensions
to doubly nonlinear evolutions, can be found in~\cite{BCD_2018,S_2008}.

\medskip

The evolution equation below is a particular doubly nonlinear equation, with dissipation potential 
\[
v\mapsto \frac1{m^\ast}\|v\|^{m^\ast}.
\]
Existence results in a more general setting are classical, see~\cite{C_1992, CV_1990,  S_2008}. 
For completeness, we give here a short proof based on a time reparametrization of the usual gradient flow.
Throughout, for any $m > 1$, we denote by $m^\ast :=m(m-1)^{-1}$ its H\"older conjugate, that is, $\dfrac{1}{m} + \dfrac{1}{m^\ast} = 1$.

\begin{lemma}\label{lem.mGF}
    Let $m > 1$, let $\mathcal H$ be a Hilbert space and let $f \in \mathcal C^1(\mathcal H)$ be a convex bounded from below function. 
    Then, for every $x \in \mathcal H$, there exists a curve $\gamma = \gamma_x \in \mathcal{C}^1((0, + \infty); \mathcal H) \cap \mathcal{C}([0, + \infty);\mathcal{H})$ such that 
    \begin{equation}\label{m-GF}
        \begin{dcases}
            \| \dot{\gamma} (t) \| ^{m^\ast - 2} \,\dot{\gamma}(t)  = - \nabla f(\gamma(t)), & t \in (0, + \infty), \medskip\\
            \gamma(0) = x.
        \end{dcases}
    \end{equation}
\end{lemma}

\begin{proof}
By the theory of gradient flows for convex functions (see  \textit{e.g.} \cite[Theorem~17.2.2]{A_Book}), there exists $u \in \mathcal{C}^1((0, + \infty); \mathcal{H}) \cap \mathcal{C}([0, + \infty); \mathcal{H})$ such that
\begin{align*}
    \begin{dcases}
        \dot{u}(t) = - \nabla f(u(t)), \quad t \in (0, + \infty) \\
        u(0) = x.
    \end{dcases}
\end{align*}
Further, the function $a(t) := \| \nabla f(u(t)) \|$ is continuous and nonincreasing on $[0, + \infty)$. 
If $a(0) = 0$, then the constant curve
$\gamma \equiv x$ satisfies~\eqref{m-GF}. 
It suffices to consider the case $a(0) > 0$.

\medskip

Set $S := \inf \{ s \geq 0 : a(s) = 0 \} \in (0,+\infty]$ and define
\begin{align*}
    \theta(s) := \int_0^s a(r)^{2-m} \, dr, \quad \text{ for } s \in [0,S).
\end{align*}
The function $\theta$ is strictly increasing and satisfies $\theta'(s) = a(s)^{2-m} > 0$ for every $s \in (0,S)$. \\
Set $T := \lim_{s \nearrow S} \theta(s) \in (0,+\infty]$ and define $\gamma(t) := u(\theta^{-1}(t))$ for $0 \leq t < T$. By the inverse function theorem, $(\theta^{-1})'(t) = a(\theta^{-1}(t))^{m-2}$ for every $t \in (0,T)$, and therefore
\begin{align*}
    \dot{\gamma}(t)
    = - a(\theta^{-1}(t))^{m-2}\nabla f(u(\theta^{-1}(t)))
    = - \|\nabla f(\gamma(t))\|^{m-2}\nabla f(\gamma(t)).
\end{align*}
Since the maps $v \mapsto \|v\|^{m^\ast-2}v$ and $v \mapsto \|v\|^{m-2}v$ are inverse to each other, this is equivalent to~\eqref{m-GF}.

\medskip

If $T = +\infty$, the proof is complete.
Assume that $T < +\infty$ and consider the following cases.

\medskip

\textit{Case 1:} $S < +\infty$.
We have $a(S) = 0$ and we set $\bar{x} := u(S)$. 

\medskip

\textit{Case 2:} $S = +\infty$. In this case, we necessarily have $1 < m < 2$.
Indeed, if $m \geq 2$, the monotonicity of $a(\cdot)$ gives $a(s)^{2-m} \geq a(0)^{2-m}$ for every $s \geq 0$, contradicting $T < +\infty$.
Recall that $a(s) \xrightarrow{s \to \infty} 0$. 
Since $1 < m < 2$, we have $a(s) \leq a(s)^{2-m}$ for all sufficiently large $s \geq s_0$ and hence
\begin{align*}
    \int_0^\infty \| \dot{u}(s) \| \, ds
    = \int_0^\infty a(s) \, ds
    \leq \int_0^{s_0} a(s) \, ds +  \int_{s_0}^\infty a(s)^{2 - m} \, ds < +\infty.
\end{align*}
Since the curve $u$ has finite length, there exists $\bar x \in \mathcal{H}$ such that $u(s) \xrightarrow{s \to \infty} \bar x$ and so $\nabla f(\bar x) = 0$.

\medskip

In both cases, $\gamma(t) \longrightarrow \bar{x}$ as $t \nearrow T$.
We extend $\gamma$ by setting
$\gamma(t) := \bar{x}$ for $t \geq T$. 
Finally,
\begin{align*}
    \| \dot{\gamma}(t) \| = a(\theta^{-1}(t))^{m-1}     \longrightarrow 0
    \quad \text{ as } \quad t \nearrow T.
\end{align*}
The extension belongs to
$C^1((0,+\infty);\mathcal H) \cap C([0,+\infty);\mathcal H)$ and solves~\eqref{m-GF}.
Lemma~\ref{lem.mGF} is proven.
\end{proof}

\begin{proposition}[Finite horizon]\label{prop.fini_hori}
    Let $m > 1$ and let $\mathcal H$ be a Hilbert space. 
    Let $f \in \mathcal C(\mathcal H)$ be bounded from below such that $V_m \in \mathcal{C}(\mathcal{H})$  is convex.
    Fix $T > 0$ and $x \in \mathcal{H}$.
    Define
    \[
        \mathcal{A}_{x, T} := \big\{ \xi \in C([0, T]; \mathcal{H}) \cap C^1((0, T); \mathcal{H}) \, : \, \xi(0) = x  \big\}
    \]
    and for $\xi \in\mathcal{A}_{x, T}$, set
    \[
            J_T(\xi) := \int_0^T \left( \dfrac{1}{m^\ast} \| \dot{\xi}(\tau) \|^{m^\ast} + \dfrac{1}{m} V_m(\xi(\tau)) \right) \, d\tau + f(\xi(T)).
    \]

    Then, $J_T$ is strictly convex on $\mathcal{A}_{x, T}$ and moreover, for any $\gamma \in \mathcal A_{x, T}$, the following assertions are equivalent:
    \begin{itemize}
        \item[(i)] $\gamma$ solves~\eqref{m-GF} on $[0, T]$;

        \item[(ii)] $\gamma$ is the unique minimizer of $J_T$ over $\mathcal{A}_{x, T}$.
    \end{itemize}
\end{proposition}

\begin{proof}
Thanks to Corollary~\ref{corol.C1+conv}, $f$ is convex and of class $\mathcal{C}^1 $ and so $V_m = \| \nabla f \|^m$.
Moreover, the convexity of $V_m$ implies that the map
\[
    \xi \longmapsto \int_0^T V_m(\xi(t)) \, dt
\]
is convex on $\mathcal{A}_{x, T}$.
Also, since the map
\[
    \xi \longmapsto \int_0^T \| \dot{\xi}(t) \|^{m^\ast} \, dt
\]
is strictly convex, $J_T$ is strictly convex on $\mathcal{A}_{x, T}$.

\medskip

Observe first that, for every $v,w \in \mathcal H$, Fenchel's inequality gives
\begin{equation}\label{fenchel}
    \frac{1}{m^\ast} \|v\|^{m^\ast}     + \frac{1}{m} \|w\|^m + \langle w,v\rangle \geq 0,
\end{equation}
with equality if and only if $\|v\|^{m^\ast-2}v + w = 0$.

\medskip 

Fix $\xi \in \mathcal{A}_{x, T}$ such that $J_T(\xi) < +\infty$.
The chain rule gives
\[
    \dfrac{d}{dt} f(\xi(t)) = \langle \nabla f(\xi(t)), \dot{\xi}(t) \rangle \quad  \text{ for every } t \in (0, T).
\]
Therefore, due to~\eqref{fenchel},
\begin{equation}\label{J_T.EST}
    J_T(\xi)
    = f(\xi(0)) + \int_0^T
    \left(
        \frac{1}{m^\ast} \| \dot{\xi}(t) \|^{m^\ast} + \frac{1}{m} \| \nabla f(\xi(t)) \|^m     + \langle \nabla f(\xi(t)),\dot{\xi}(t)\rangle
    \right) \, dt  \geq f(x).
\end{equation}
If $J_T(\xi) = + \infty$, the above inequality is immediate.
\medskip

$(i) \Longrightarrow (ii)$.
Assume that $\gamma$ solves~\eqref{m-GF}.
Applying~\eqref{fenchel} to the case $v = \dot{\gamma}(t)$ and $w = \nabla f(\gamma(t))$, equality of~\eqref{fenchel} yields $J_T(\gamma) = f(\gamma(0)) = f(x)$.
It follows from~\eqref{J_T.EST} that $\gamma$ minimizes $J_T$ on $\mathcal{A}_{x, T}$.
Moreover, since $J_T$ is strictly convex, $\gamma$ is the unique minimizer of $J_T$ on $\mathcal{A}_{x, T}$.

\medskip

$(ii) \Longrightarrow (i)$.
Conversely, assume that $\gamma$ satisfies $(ii)$. 
Thanks to Lemma~\ref{lem.mGF}, there exists a solution
$\gamma_\ast$ of~\eqref{m-GF} with $\gamma_\ast(0) = x$.
From the previous part, we know that
$J_T(\gamma_\ast) = f(x)$ and hence, by~\eqref{J_T.EST}, one has
\begin{align*}
    f(\gamma(0)) \leq J_T(\gamma)    \leq J_T(\gamma_\ast) = f(\gamma(0)).
\end{align*}
Consequently,
\begin{align*}
    \int_0^T
    \left(
        \frac{1}{m^\ast} \| \dot{\gamma}(t) \|^{m^\ast}
        + \frac{1}{m} \| \nabla f(\gamma(t)) \|^m
        + \langle \nabla f(\gamma(t)),\dot{\gamma}(t)\rangle
    \right) \, dt
    = 0.
\end{align*}
The integrand is nonnegative and continuous on $(0,T)$, hence it vanishes there.
Equality in~\eqref{fenchel} implies that
\begin{align*}
    \| \dot{\gamma}(t) \|^{m^\ast-2}\dot{\gamma}(t)
    = - \nabla f(\gamma(t)),
    \quad \text{ for every } t \in (0,T).
\end{align*}
Thus $\gamma$ solves~\eqref{m-GF}, completing the proof.
\end{proof}

\begin{proposition}[Infinite horizon]
Let $m > 1$ and let $\mathcal H$ be a Hilbert space. 
    Let $f \in \mathcal C(\mathcal H)$ be bounded from below such that $V_m \in \mathcal{C}(\mathcal{H})$  is convex.
    Fix $x \in \mathcal{H}$.
    Define
    \[
        \mathcal{A}_{x, \infty} := \big\{ \xi \in C([0, + \infty); \mathcal{H}) \cap C^1((0, + \infty); \mathcal{H}) \, : \, \xi(0) = x \big\}
    \]
    and for $\xi \in \mathcal{A}_{x, \infty}$, set
    \[
       J_\infty(\xi) := \int_0^\infty\left( \dfrac{1}{m^\ast} \| \dot{\xi}(\tau) \|^{m^\ast} + \dfrac{1}{m} V_m(\xi(\tau)) \right) \, d\tau. 
    \]
    Then, $J_\infty$ is strictly convex on $\mathcal{A}_{x, \infty}$ and moreover, for any $\gamma \in \mathcal{A}_{x, \infty}$ with $J_\infty(\gamma) < + \infty$, the following assertions are equivalent:
    \begin{itemize}
        \item[(i)] $\gamma$ solves~\eqref{m-GF} on $[0, + \infty)$ with $\gamma(0) = x$;

        \item[(ii)] $\gamma$ is the unique minimizer of $J_\infty$ over $\mathcal{A}_{x, \infty}$.
    \end{itemize}
\end{proposition}

\begin{proof}
Thanks to Corollary~\ref{corol.C1+conv}, $f$ is convex and of class $\mathcal C^1$ and so $V_m = \| \nabla f \|^m$.
Reasoning as in the proof of Proposition~\ref{prop.fini_hori}, we know that $J_\infty$ is strictly convex on $\mathcal{A}_{x, \infty}$.
Let us start with the following claim.

\begin{claim}\label{claim.2}
If $\xi \in C([0,+\infty);\mathcal H) \cap C^1((0,+\infty);\mathcal H)$ satisfies $J_\infty(\xi) < +\infty$, then
\begin{align*}
    \lim_{t \to +\infty} f(\xi(t)) = \inf f.
\end{align*}
\end{claim}

\textit{Proof of Claim~\ref{claim.2}.}
Since $J_\infty(\xi) < + \infty$, we note that $\dot{\xi} \in L^{m^\ast}(\R_+;\mathcal H)$ and $\nabla f(\xi(\cdot)) \in L^m(\R_+;\mathcal H)$. 
Hence, by H\"older's inequality,
\begin{align*}
    \int_0^\infty \big| \langle \nabla f(\xi(t)),\dot{\xi}(t)\rangle \big| \, dt  \leq \|\nabla f(\xi(\cdot))\|_{L^m} \|\dot{\xi}\|_{L^{m^\ast}} < +\infty.
\end{align*}
Notice that, by the chain rule, $\dfrac{d}{dt}f(\xi(t)) = \langle \nabla f(\xi(t)),\dot{\xi}(t)\rangle$ for every $t > 0$. 
It follows that the  limit $\ell := \lim_{t \to +\infty} f(\xi(t))$ exists and $\ell \geq \inf f$.

\medskip

It remains to prove that $\ell \leq \inf \, f$.
Since $\|\nabla f(\xi(\cdot))\|^m \in L^1(\R_+; \R_+)$, there exists a sequence $t_n \to +\infty$ such that
$t_n \|\nabla f(\xi(t_n))\|^m \longrightarrow 0$ (see e.g. Claim~\ref{claim.tn}). 
Fix $y \in \mathcal H$. 
By convexity of $f$ and H\"older's inequality, we have
\begin{align*}
    f(\xi(t_n)) - f(y) & \leq \|\nabla f(\xi(t_n))\| \|\xi(t_n)-y\| \\
    &\leq \|\nabla f(\xi(t_n))\|
    \left(
        \|\xi(0)-y\| + t_n^{1/m}\|\dot{\xi}\|_{L^{m^\ast}}
    \right).
\end{align*}
Recall that  $\|\nabla f(\xi(t_n))\| \xrightarrow{n \to \infty} 0$. 
Letting $n \to \infty$, we infer that $\ell \leq f(y)$. 
Since $y$ is arbitrary, $\ell \leq \inf f$ and therefore $\ell = \inf f$.
\hfill$\Diamond$

\medskip

Fix $\xi \in \mathcal{A}_{x, \infty}$ such that $J_\infty(\xi) < +\infty$.
Applying the chain rule on $[0,T]$ and then letting $T \to +\infty$, Claim~\ref{claim.2} and~\eqref{fenchel} yield
\begin{equation}\label{J_infty.EST}
\begin{split}
    J_\infty(\xi)
    & = f(\xi(0)) - \inf f
    + \int_0^\infty
    \left(
        \frac{1}{m^\ast}\|\dot{\xi}(t)\|^{m^\ast} + \frac{1}{m}\|\nabla f(\xi(t))\|^m + \langle \nabla f(\xi(t)),\dot{\xi}(t)\rangle
    \right) \, dt \\
    & \geq f(\xi(0)) - \inf f = f(x) - \inf f.
\end{split}
\end{equation}
The same inequality is trivial when $J_\infty(\xi) = +\infty$.

\medskip

$(i) \Longrightarrow (ii)$.
Assume that $\gamma$ solves~\eqref{m-GF}. Equality in~\eqref{fenchel} gives, by~\eqref{J_infty.EST},
\begin{align*}
    J_\infty(\gamma) = f(\gamma(0)) - \inf f = f(x) - \inf f.
\end{align*}
It follows again from~\eqref{J_infty.EST} that $J_\infty(\gamma) \leq J_\infty(\xi)$ for every $\xi \in \mathcal{A}_{x, \infty}$. Then, since $J_\infty$ is strictly convex, $\gamma$ is the unique minimizer of $J_\infty$ on $\mathcal{A}_{x, \infty}$.

\medskip

$(ii) \Longrightarrow (i)$.
Conversely, assume that $\gamma$ satisfies $(ii)$. 
Thanks to Lemma~\ref{lem.mGF}, there exists a solution $\gamma_\ast$ of~\eqref{m-GF} with $\gamma_\ast(0) = x$.
Equation~\eqref{m-GF} yields 
\[
    \|\dot{\gamma}_\ast(t)\|^{m^\ast} = \|\nabla f(\gamma_\ast(t))\|^m \quad \text{ and } \quad \dfrac{d}{dt}f(\gamma_\ast(t)) = -\|\nabla f(\gamma_\ast(t))\|^m.
\]
Therefore, for every $T > 0$,
\begin{align*}
    \int_0^T
    \left(
        \frac{1}{m^\ast}\|\dot{\gamma}_\ast(t)\|^{m^\ast} + \frac{1}{m}\|\nabla f(\gamma_\ast(t))\|^m
    \right) \, dt = f(x) - f(\gamma_\ast(T))
    \leq f(x) - \inf f.
\end{align*}
Thus $J_\infty(\gamma_\ast) < +\infty$. 
Further, thanks to Claim~\ref{claim.2}, we have $f(\gamma_\ast(T)) \xrightarrow{T \to + \infty} \inf f$ and so $ J_\infty(\gamma_\ast) = f(x) - \inf f$.
By~\eqref{J_infty.EST} and the minimality of $\gamma$, we obtain
\begin{align*}
    f(x) - \inf f \leq J_\infty(\gamma) \leq J_\infty(\gamma_\ast) = f(x) - \inf f. 
\end{align*}
Therefore,
\begin{align*}
    \int_0^\infty
    \left(
        \frac{1}{m^\ast}\|\dot{\gamma}(t)\|^{m^\ast}
        + \frac{1}{m}\|\nabla f(\gamma(t))\|^m
        + \langle \nabla f(\gamma(t)),\dot{\gamma}(t)\rangle
    \right) \, dt = 0.
\end{align*}
The integrand is nonnegative and continuous on $(0,+\infty)$; hence it vanishes there. 
Equality in~\eqref{fenchel} implies that
\begin{align*}
    \|\dot{\gamma}(t)\|^{m^\ast-2}\dot{\gamma}(t)
    = -\nabla f(\gamma(t)), \quad \text{ for every } t > 0.
\end{align*}
Thus $\gamma$ solves~\eqref{m-GF}, completing the proof.
\end{proof}

\subsection{A uniqueness theorem for eikonal equations}
Let $X$ be a Banach space and $\ell:X\to[0,+\infty)$ be a continuous function with $\inf_X\ell=0$.
Consider the eikonal equation
\begin{equation}\label{eiko}
    \begin{dcases}
        \, \|Df(x)\|_\ast=\ell(x) \quad \text{in } X,\\
        \, \inf_X f=0,
    \end{dcases}
\end{equation}
where $\|\cdot\|_\ast$ denotes the dual norm.
Viscosity solutions for Hamilton--Jacobi equations in Banach spaces, under additional assumptions such as the Radon--Nikod\'ym property, have been studied since the eighties, see~\cite{CL_1985, CL_1986}.
Extensions of the theory to general metric spaces can be found in~\cite{ GS_2015, GHN_2015, LSZ_2021}.

\medskip

By a \emph{Monge solution of~\eqref{eiko}}, we mean a continuous function $f:X\to\R$ such that
\begin{equation}\label{monge}
    \begin{dcases}
        \, s_f(x)=\ell(x) \quad \text{in } X,\\
        \, \inf_X f=0.
    \end{dcases}
\end{equation}
Consistency with the classical viscosity notion in the Euclidean setting is discussed in~\cite[Section~6]{GHN_2015}, while the relation with other notions of viscosity solutions in metric spaces is studied in~\cite{LSZ_2021}, see in particular~\cite[Remark~4.9]{LSZ_2021}.

\begin{proposition}\label{prop.conv.deter}
Let $X$ be a Banach space and let $\ell:X\to[0,+\infty)$ be continuous.
Assume in addition that $\ell^m$ is convex for some $m\geq1$.
Then there exists at most one Monge solution of~\eqref{monge}.
\end{proposition}

\begin{proof}
If $f$ is a Monge solution of~\eqref{monge}, then $s_f^m = \ell^m$ is convex and Theorem~\ref{thm.conti-case} yields that $f$ is convex.
If $f$ and $g$ are two Monge solutions, then $s_f=s_g$ and the determination result of~\cite{TZ_2023} yields $f = g + c$ for some $c\in\R$.
Since $\inf_X f = \inf_X g = 0$, we obtain $c = 0$.
\end{proof}

\begin{remark}
Proposition~\ref{prop.conv.deter} provides a uniqueness result. Let us also record, within the framework of Proposition~\ref{prop.conv.deter}, a simple existence result for Monge solutions of~\eqref{monge}.
Assume that $Z:=\{\ell=0\}\neq\emptyset$ and define
\[
f(x):=\inf\left\{ \int_0^T \ell(\gamma(t))\,dt : T\geq 0,\  \gamma:[0,T]\to X \text{ is $1$--Lipschitz},\ \gamma(0)=x,\ \gamma(T)\in Z \right\}.
\]
Then $f$ is finite and locally Lipschitz, $f = 0$ on $Z$ and $f$ is a Monge solution of~\eqref{monge}.

\medskip

Indeed, fix $x,y\in X$. Concatenating the segment from $x$ to $y$ with an almost minimizing curve starting from $y$ gives
\[
    f(x)-f(y) \leq \|x-y\| \sup_{z\in[x,y]}\ell(z).
\]
Exchanging $x$ and $y$ shows that $f$ is locally Lipschitz and letting $y$ tend to $ x$ yields $s_f(x) \leq \ell(x)$.

\medskip

Conversely, let $x\notin Z$ and choose $r> 0$ sufficiently small so that $\overline{B}_r(x)\cap Z = \varnothing$.
For $\varepsilon > 0$, let $\gamma$ be an $\varepsilon$--minimizer in the definition of $f(x)$ and let $\tau$ be the exit time from $B_r(x)$ of $\gamma$.
Setting $y := \gamma(\tau)$, we have $\|x-y\| = r$ and $\tau \geq r$.
By the definition of $f(y)$,
\[
f(x)-f(y) \geq \int_0^\tau \ell(\gamma(t))\,dt-\varepsilon \geq r\inf_{B_r(x)} \ell-\varepsilon.
\]
Choosing $\varepsilon = o(r)$ and letting $r \searrow 0$, the continuity of $\ell$ yields $s_f(x) \geq \ell(x)$.
Hence $s_f(x) = \ell(x)$ for every $x \notin Z$.
If $x\in Z$, then $f(x) = 0$ and since $f\geq0$, we have $s_f(x) = 0 = \ell(x)$.
Therefore $s_f = \ell$ on~$X$.

\medskip

We refer to~\cite[Theorem~4.2]{GHN_2015} for a related value-function construction for eikonal equations on metric spaces and to~\cite[Remark~4.9]{LSZ_2021} for its relation with the Monge solution.
\end{remark}

\vspace{0.7cm}

\textbf{Acknowledgement.} This research was initiated during a research visit of the first author to VADOR, TU Wien (March 2026). This author thanks his hosts for hospitality. The second author thanks Gilles Godefroy for the proof mentioned in Remark~\ref{rem-god} (concerning \cite[Theorem~6]{M_1964}). The authors thank Alberto Dom\'inguez Corella and Sebasti\'an Tapia Garc\'ia for useful discussions. The research of the second author was partially supported by the Austrian Science Fund (Grant FWF 10.55776/P36344) and by the SABOCPR project ANR-25-CE40-3469-01 and FWF 4368225. The research of the third author was funded by the Austrian Science Fund (FWF) (10.55776/STA223).
\smallskip\newline
For open access purposes, the second author has applied a CC BY public copyright license to any author-accepted manuscript version arising from this submission.



\vspace{0.5cm}

\noindent Tahar Zamene BOULMEZAOUD
\medskip

\noindent 
Laboratoire de Math\'ematiques de Versailles\\
Universit\'e de Versailles Saint-Quentin-en-Yvelines - Universit\'e Paris-Saclay\\
45, avenue des Etats-Unis, 78035, Versailles, Cedex, France,\smallskip\newline
and  \\
Department of Mathematics and Statistics, University of Victoria, Victoria, British
Columbia, Canada. 
\noindent E-mail: \texttt{tahar.boulmezaoud@uvsq.fr}       \newline\noindent\texttt{https://boulmezaoud.perso.math.cnrs.fr/}

\vspace{0.5cm}

\noindent Aris DANIILIDIS 
\medskip

\noindent Institut f\"{u}r Stochastik und Wirtschaftsmathematik, VADOR E105-04
\newline TU Wien, Wiedner Hauptstra{\ss }e 8, A-1040 Wien\medskip
\newline\noindent E-mail: \texttt{aris.daniilidis@tuwien.ac.at}
\newline\noindent\texttt{https://www.arisdaniilidis.at/}
\smallskip\newline\noindent
\noindent Research supported by the Austrian \texttt{FWF} grant \texttt{DOI 10.55776/P-36344N} and by \smallskip \newline the French-Austrian \texttt{SABOCPR} project \texttt{ANR-25-CE40-3469-01} and \texttt{FWF 4368225}. 

\vspace{0.5cm}

\noindent Tr\'i Minh L\^E
\medskip 

\noindent  Faculty of Mathematics,
University of Vienna
\newline Oskar-Morgenstern-Platz 1, 1090 Wien
\medskip
\newline\noindent E-mail: \texttt{tri.minh.le@univie.ac.at}
\newline\noindent\texttt{https://sites.google.com/view/tri-minh-le}
\smallskip\newline\noindent
\noindent Research supported by the Austrian \texttt{FWF} grant \texttt{DOI 10.55776/STA223}.


\begin{thebibliography}{99}

\bibitem{AGS_2008}
\textsc{L. Ambrosio, N. Gigli and G. Savar\'e,} {\it Gradient flows in metric spaces and in the space of probability measures}, second edition, 
Lectures in Mathematics ETH Z\"urich, Birkh\"auser, Basel, 2008.

\bibitem{A_Book}
\textsc{H. Attouch, G. Buttazzo and G. Michaille},
\textit{Variational analysis in {S}obolev and {BV} spaces},
Applications to PDEs and optimization. Second edition.
MOS-SIAM Ser. Optim., 17, 2014.

\bibitem{BCD_2018}
\textsc{T. Z. Boulmezaoud, P. Cieutat and A. Daniilidis},
Gradient flows, second-order gradient systems and convexity,
\emph{SIAM J. Optim.} \textbf{28} (2018), 2049–-2066.

\bibitem{B_73}
\textsc{H. Br\'ezis},
\emph{Op\'erateurs maximaux monotones et semi-groupes de             contractions dans les espaces de {H}ilbert},
North-Holland Math. Stud., No. 5, 1973.

\bibitem{BE_76}
\textsc{H. Br\'ezis and I. Ekeland},
Un principe variationnel associ\'e{} \`a{} certaines             \'equations paraboliques. {L}e cas ind\'ependant du temps,
\emph{C. R. Acad. Sci. Paris S\'er. A-B} \textbf{282} (1976), A971-–A974.

\bibitem{BE_76_02}
\textsc{H. Br\'ezis and I. Ekeland}
Un principe variationnel associ\'e{} \`a{} certaines             \'equations paraboliques. {L}e cas d\'ependant du temps,
\emph{C. R. Acad. Sci. Paris S\'er. A-B} \textbf{282} (1976), A1197--A1198.



\bibitem{B_94}
\textsc{A. M. Bruckner,}
\emph{Differentiation of real functions},
CRM Monograph Series, vol. 5, American Mathematical Society,
Providence, RI, 1994.

\bibitem{CC_2010}
\textsc{L. A. Caffarelli and M. G. Crandall},
Distance functions and almost global solutions of eikonal equations,
\emph{Comm. Partial Differential Equations} \textbf{35} (2010), 391–-414.

\bibitem{C_1992}
\textsc{P. Colli},
On some doubly nonlinear evolution equations in Banach spaces,
\emph{Japan J. Indust. Appl. Math.} \textbf{9} (1992), 181--203.

\bibitem{CV_1990}
\textsc{P. Colli and A. Visintin},
On a class of doubly nonlinear evolution equations,
\emph{Comm. Partial Differential Equations} \textbf{15} (1990), 737--756.


\bibitem{CL_1985}
\textsc{M. G. Crandall and P.-L. Lions},
{H}amilton-{J}acobi equations in infinite dimensions. {I}. {U}niqueness of viscosity solutions,
\emph{J. Funct. Anal.} \textbf{62} (1985), 379--396.


\bibitem{CL_1986}
\textsc{M. G. Crandall and P.-L. Lions},
Hamilton-{J}acobi equations in infinite dimensions. {II}. {E}xistence of viscosity solutions,
\emph{J. Funct. Anal.} \textbf{65} (1986), 368--405.

\bibitem{DD_2023}
\textsc{A. Daniilidis and D. Drusvyatskiy}, 
The slope robustly determines convex functions, 
\emph{Proc. Amer. Math. Soc.} {\bf 151} (2023), 4751--4756.

\bibitem{DST_2025} 
\textsc{A. Daniilidis, D. Salas and S. Tapia-Garc\'ia}, 
A slope generalization of {A}ttouch theorem, 
\emph{Math. Program.} {\bf 212} (2025), 319--348.


\bibitem{GMT_1980}
\textsc{E. De~Giorgi, A. Marino and M. Tosques}, 
Problems of evolution in metric spaces and maximal decreasing curve, 
\emph{Atti Accad. Naz. Lincei Rend. Cl. Sci. Fis. Mat. Nat.} (8) {\bf 68} (1980), no.~3, 180--187.

\bibitem{GS_2015}
\textsc{W. Gangbo and A. \'{S}wi\k{e}ch},
Metric viscosity solutions of {H}amilton--{J}acobi equations depending on local slopes,
\emph{Calc. Var. Partial Differential Equations} \textbf{54} (2015), 1183--1218.


\bibitem{GHN_2015}
\textsc{Y. Giga, N. Hamamuki and A. Nakayasu},
Eikonal equations in metric spaces,
\emph{Trans. Amer. Math. Soc.}
\textbf{367} (2015) 49–-66.


\bibitem{God}
\textsc{G. Godefroy}, \emph{Introduction aux m\'ethodes de Baire},
Calvage \& Mounet, Paris, 2022.


\bibitem{HT_2006}
\textsc{J. W. Hagood and B. S. Thomson},
Recovering a function from a Dini derivative,
\emph{Am. Math. Monthly} \textbf{113} (2006), 34–-46.
\bibitem{I_2025}

\textsc{R. Ignat},
A short proof of the {$\mathcal{C}^{1, 1}$} regularity for the eikonal equation,
\emph{C. R. Math. Acad. Sci. Paris} \textbf{363} (2025),  887–-891.

\bibitem{LSZ_2021}
\textsc{Q. Liu, N. Shanmugalingam and X. Zhou}, 
Equivalence of solutions of eikonal equation in metric spaces,
\emph{J. Differential Equations} \textbf{272} (2021), 979–-1014.

\bibitem{LW_2026}
\textsc{Q. Liu and M.~B.~P. Wiranata}, Monge solutions of time-dependent Hamilton-Jacobi equations in metric spaces, 
\emph{ESAIM Control Optim. Calc. Var.} {\bf 32} (2026), Paper No. 2.

\bibitem{M_1964}
\textsc{M. Mehdi}, On convex functions, \emph{J. London Math. Soc.} {\bf 39} (1964), 321--326.

\bibitem{PSV_2021}
\textsc{P. P\'erez-Aros, D. Salas and E. Vilches}, 
Determination of convex functions via subgradients of minimal norm, \emph{Math. Program.} {\bf 190} (2021), 561--583.

\bibitem{S_2008}
\textsc{U. Stefanelli},
The {B}r\'ezis--{E}keland principle for doubly nonlinear equations,
\emph{SIAM J. Control Optim.} \textbf{47} (2008), 1615--1642.

\bibitem{TZ_2023}
\textsc{L. Thibault and D. Zagrodny}, 
Determining functions by slopes, 
\emph{Commun. Contemp. Math.} {\bf 25} (2023), Paper No. 2250014.

\bibitem{V_2021}
\textsc{E. Vilches}, Proximal determination of convex functions, \emph{J. Convex Anal.} {\bf 28} (2021), 1187--1192.




\end{thebibliography}
\end{document}